\documentclass[12pt]{article}%
\usepackage{amsmath}
\usepackage{amsfonts}
\usepackage{amssymb}
\usepackage{graphicx}%
\providecommand{\U}[1]{\protect\rule{.1in}{.1in}}
\newtheorem{theorem}{Theorem}

\newtheorem{corollary}[theorem]{Corollary}

\newtheorem{definition}[theorem]{Definition}
\newtheorem{example}[theorem]{Example}

\newtheorem{lemma}[theorem]{Lemma}

\newtheorem{proposition}[theorem]{Proposition}
\newtheorem{remark}[theorem]{Remark}

\newenvironment{proof}[1][Proof]{\noindent\textbf{#1.} }{\ \rule{0.5em}{0.5em}}
\begin{document}

\title{Efficiency analysis for the Perron vector of a reciprocal matrix}
\author{Susana Furtado\thanks{Email: sbf@fep.up.pt. orcid.org/0000-0003-0395-5972. The
work of this author was supported by FCT- Funda\c{c}\~{a}o para a Ci\^{e}ncia
e Tecnologia, under project UIDB/04561/2020.} \thanks{Corresponding author.}\\CMAFcIO and Faculdade de Economia \\Universidade do Porto\\Rua Dr. Roberto Frias\\4200-464 Porto, Portugal
\and Charles R. Johnson \thanks{Email: crjohn@wm.edu. The work of this author was
supported in part by National Science Foundation grant DMS-0751964.}\\Department of Mathematics\\College of William and Mary\\Williamsburg, VA 23187-8795}
\maketitle

\begin{abstract}
In prioritization schemes, based on pairwise comparisons, such as the
Analytical Hierarchy Process, it is necessary to extract a cardinal ranking
vector from a reciprocal matrix that is unlikely to be consistent. It is
natural to choose such a vector only from efficient ones. One of the most used
ranking methods employs the (right) Perron eigenvector of the reciprocal
matrix as the vector of weights. It is known that the Perron vector may not be
efficient. Here, we focus on extending arbitrary reciprocal matrices and show,
constructively, that two different extensions of any fixed size always exist
for which the Perron vector is inefficient and for which it is efficient, with
the following exception. If $B$ is consistent, any reciprocal matrix obtained
from $B$ by adding one row and one column has efficient Perron vector. As a
consequence of our results, we obtain families of reciprocal matrices for
which the Perron vector is inefficient. These include known classes of such
matrices and many more. We also characterize the $4$-by-$4$ reciprocal
matrices with inefficient Perron vector. Some prior results are generalized or completed.

\end{abstract}



\section{Introduction}

The Analytic Hierarchy process was introduced in \cite{saaty1977} and is used
in decision making. It is based upon "reciprocal matrices" that represent
pair-wise ratio comparisons among several alternatives. Such matrices also
arise in other multi-criterion decision making models.

An $n$-by-$n$ entry-wise positive matrix $A=\left[  a_{ij}\right]  $ is called
\emph{reciprocal} if $a_{ji}=\frac{1}{a_{ij}},$ when $1\leq i,j\leq n.$ We
denote by $\mathcal{PC}_{n}$ the set of all such matrices. A matrix
$A\in\mathcal{PC}_{n}$ is further said to be \emph{consistent} if
$a_{ij}a_{jk}=a_{ik}$ for all $i,j,k$ (otherwise it is \emph{inconsistent}).
This is the case if and only if there is a positive vector $w=\left[
\begin{array}
[c]{ccc}%
w_{1} & \ldots & w_{n}%
\end{array}
\right]  ^{T}$ such that $a_{ij}=\frac{w_{i}}{w_{j}}$ for all $i,j.$ Such a
vector $w$ is unique up to a factor of scale and cardinally ranks the
alternatives. Any matrix in $\mathcal{PC}_{2}$ is consistent. When $n>2,$
consistency of the ratio comparisons is unlikely. However, a cardinal ranking
vector, also called a weight vector, should still be obtained from a
reciprocal matrix \cite{choo, dij, golany}.

Many ways of constructing a weight vector from a reciprocal matrix $A$ have
been proposed. The classical proposal for such a vector is the (right) Perron
vector \cite{saaty1977, Saaty}. Other proposals include the (entry-wise)
geometric mean of the columns \cite{blanq2006, fichtner86} and any weighted
geometric mean of columns of a reciprocal matrix \cite{FJ3}.

An important property that a weight vector obtained from a reciprocal matrix
should have is efficiency (also called Pareto optimality). Denote by
$\mathcal{V}_{n}$ the set of positive $n$-vectors. A vector $w=\left[
\begin{array}
[c]{ccc}%
w_{1} & \ldots & w_{n}%
\end{array}
\right]  ^{T}\in\mathcal{V}_{n}$ is called \emph{efficient} for $A=\left[
a_{ij}\right]  \in\mathcal{PC}_{n}$ \cite{blanq2006} if, for every other
positive vector $v=\left[
\begin{array}
[c]{ccc}%
v_{1} & \ldots & v_{n}%
\end{array}
\right]  ^{T}\in\mathcal{V}_{n},$%
\[
\left\vert a_{ij}-\frac{v_{i}}{v_{j}}\right\vert \leq\left\vert a_{ij}%
-\frac{w_{i}}{w_{j}}\right\vert \text{, for all }i,j=1,\ldots n,
\]
implies $v$ is a positive multiple of $w,$ i.e. no other consistent matrix
approximating $A$ is unambiguously better than the one associated with $w$.
(Above $|\cdot|$ denotes the absolute value of a real number.) We denote the
set of all efficient vectors for $A\ $ by $\mathcal{E}(A).$

The efficient vectors for a consistent matrix are the positive multiples of
any of its columns (projectively unique). When $A$ is inconsistent, there are
infinitely many (non-proportional) efficient vectors for $A.$ Clearly, any
positive multiple of an efficient vector for $A$ is still efficient.

Vector efficiency has been widely studied for several years. It is known that
a vector $w$ is efficient for $A$ if and only if a certain directed graph
(digraph) $G(A,w),$ constructed from $A$ and $w,$ is strongly connected
\cite{blanq2006} (see Section \ref{sdig}). Any weighted geometric mean of
columns of a reciprocal matrix is efficient \cite{blanq2006, FJ3}. The Perron
vector of a reciprocal matrix may or may not be efficient. Numerical studies
show that the Perron vector is often inefficient in low dimensions but that
this quickly becomes rare in higher dimensions (see acknowledgement). A
structured class of reciprocal matrices for which inefficiency of the Perron
vector occurs was given in \cite{bozoki2014}. In \cite{p6, p2} the authors
show that the Perron vector of some perturbed consistent matrices, more
precisely, reciprocal matrices obtained from a consistent matrix by changing
at most two entries above the main diagonal (and the corresponding reciprocal
entries), is efficient. In particular, any matrix in $\mathcal{PC}_{3}$ has
efficient Perron vector, as it is obtained from a consistent matrix by
changing a pair of reciprocal entries \cite{FJ1}. When three pairs of
reciprocal entries are changed, inefficiency may occur (see \cite{FerFur}).
Other developments have been made concerning efficiency of a vector for a
reciprocal matrix. In \cite{CFF, Fu22}, all efficient vectors for the
perturbed consistent matrices mentioned above have been described. Recently, a
method to generate inductively all efficient vectors for a reciprocal matrix
was provided \cite{FJ2}. Several other aspects of efficiency have been studied
(see \cite{anh, baj, blanq2006, bozoki2014, european}).

Incomplete reciprocal matrices appear in many decision problems. The unknown
entries should be estimated in order to obtain a complete reciprocal matrix
from which a weight vector is extracted. When using the Perron vector as the
weights, it is important that the Perron vector of the completed matrix is
efficient. In this paper we show that any reciprocal matrix $A$ can be
extended to one whose Perron vector is efficient, by adding one row (and the
reciprocal column). If $A$ is inconsistent, an extension with inefficient
Perron vector also exists. However, if $A$ is consistent, we show that any
extension of $A$ resulting from adding one row (and the reciprocal column) has
efficient Perron vector, though inefficiency can always occur by adding two
(or more) rows and columns. We give procedures to construct such extensions.
Structured classes of reciprocal matrices with inefficient Perron vector are
provided. These include the one in \cite{bozoki2014}. In addition, we give a
characterization of all $4$-by-$4$ reciprocal matrices whose Perron vector is inefficient.

The paper is organized as follows. In Section \ref{s2} we introduce some
notation and known results that will be helpful throughout. Some new related
observations are also made. In particular, a converse of a result in
\cite{FJ2} is presented (Theorem \ref{thind}). In Section \ref{sectech} we
give some technical lemmas that will be used in the proofs of the main
results. In Section \ref{sext} we show that there is one, and only one,
reciprocal matrix $A\in\mathcal{PC}_{n}$ with prescribed Perron vector and a
given principal submatrix $B\in\mathcal{PC}_{n-1}$ (Theorem \ref{t1})$.$ In
Section \ref{s5} we show that any matrix $B\in\mathcal{PC}_{n-1}$ can be
extended to a matrix in $\mathcal{PC}_{n}$ with inefficient Perron vector,
unless $B$ is consistent, in which case we prove that such an extension is not
possible (Theorem \ref{t6}), generalizing some recent results. We start by
studying the possible extensions to a matrix with constant row sums and
inefficient Perron vector (Theorem \ref{t2}) and then, based on this result,
give a procedure to construct reciprocal matrices of arbitrary size with a
prescribed principal submatrix and inefficient Perron vector. We also
characterize the matrices in $\mathcal{PC}_{4}$ with inefficient Perron vector
(Theorem \ref{th4by4}). In Section \ref{s6} we show constructively that any
matrix $B\in\mathcal{PC}_{n-1}$ can be extended to a matrix in $\mathcal{PC}%
_{n}$ with efficient Perron vector (Theorem \ref{t7}). Several examples
illustrating the theoretical results are provided. We conclude in Section
\ref{s8} with a summary and some remarks.

\section{Notation and basic lemmas\label{s2}}

\subsection{Notation}

We start by introducing some additional notation used throughout. We denote by
$M_{n}$ the set of all $n$-by-$n$ real matrices.

For $A=[a_{ij}]\in M_{n},$ the principal submatrix of $A$ determined by
deleting (by retaining) the rows and columns indexed by a subset
$K\subset\{1,\ldots,n\}$ is denoted by $A(K)$ $(A[K]);$ we abbreviate
$A(\{i\})$ as $A(i).$ Similarly, if $w$ is a vector$,$ we denote by $w(K)$ the
vector obtained from $w$ by deleting the entries indexed by $K$ and abbreviate
$w(\{i\})$ as $w(i)$. Note that, if $A$ is reciprocal (consistent) then so are
$A(K)$ and $A[K].$

By $\mathbf{J}_{m,n}$ we denote the $m$-by-$n$ matrix with all entries equal
to $1.$ We write $\mathbf{J}_{n}$ for $\mathbf{J}_{n,n}$ and $\mathbf{e}_{n}$
for the column vector $\mathbf{J}_{n,1}.$ By $I_{n}$ we denote the identity
matrix of size $n$.

Given a vector $w,$ we denote by $\operatorname*{diag}(w)$ the diagonal matrix
with diagonal $w.$ If $w$ is positive, we say that $\operatorname*{diag}(w)$
is a positive diagonal matrix.

\subsection{Reciprocal matrices and the Perron vector\label{srec}}

We recall from Perron-Frobenius theory that the spectral radius of a positive
square matrix $A$ is a simple eigenvalue of $A\ $and there is a positive
associated (right) eigenvector, which is called the \emph{Perron vector} of
$A$ (with a possible normalization as, for example, having the last entry
equal to $1$) \cite{HJ}. In fact here, for convenience, we refer to the Perron
vector of $A$ as any positive eigenvector of $A$ (all positive eigenvectors of
$A\ $are positive multiples of one another).

\bigskip The following lemma, stated in the context of reciprocal matrices,
can be easily verified. Observe that, if $A\in\mathcal{PC}_{n}$ is subjected
to either a positive diagonal similarity or a permutation similarity, or both
(a monomial similarity), we get a matrix $A^{\prime}$ in $\mathcal{PC}_{n}.$
Moreover, if $A\ $is consistent then so is $A^{\prime}.$

\begin{lemma}
\label{ldiagon}Let $A\in\mathcal{PC}_{n}$ with Perron vector $w.$ Let $D\in
M_{n}$ be a positive diagonal matrix and let $P\in M_{n-1}$ be a permutation
matrix. Let $S$ be either $D$ or $P\oplus\left[  1\right]  .$ Then $Sw$ is the
Perron vector of $SAS^{-1}.$ Moreover, $\left(  SAS^{-1}\right)
(n)=S(n)A(n)S^{-1}(n)$. If $w(n)$ is the Perron vector of $A(n)$ then
$(Sw)(n)$ is the Perron vector of $\left(  SAS^{-1}\right)  (n).$
\end{lemma}

From the next known lemma, whose proof we include for completeness, it follows
the important fact that any reciprocal matrix is diagonally similar to a
unique reciprocal matrix with constant row sums.

\begin{lemma}
\label{ldiag}For any positive matrix $A\in M_{n},$ there is a unique (up to a
positive scalar multiple) positive diagonal matrix $D\in M_{n}$ such that
$DAD^{-1}$ has Perron vector $\mathbf{e}_{n},$ that is, has constant row sums.
\end{lemma}

\begin{proof}
Let $w$ be the Perron vector of $A$ and $D^{-1}=\operatorname*{diag}(w).$
Since $Dw=\mathbf{e}_{n}$ and, by Lemma \ref{ldiagon}, $Dw$ is the Perron
vector of $DAD^{-1},$ the existence of $D$ follows. As for the uniqueness,
suppose that, for some positive diagonal matrix $D^{\prime},$ the matrix
$D^{\prime}A(D^{\prime})^{-1}$ has Perron vector $\mathbf{e}_{n}.$ Since
$D^{\prime}w$ is the Perron vector of $D^{\prime}A(D^{\prime})^{-1},$ it
follows that $D$ and $D^{\prime}$ are equal (up to a positive multiple).
\end{proof}

\bigskip In what follows we give a sharp lower bound for the sum of all the
entries of a reciprocal matrix.

\begin{lemma}
\label{ll1}The sum of the entries of $A\in\mathcal{PC}_{n}$ is at least
$n^{2},$ with equality if and only if $A=\mathbf{J}_{n}.$
\end{lemma}

\begin{proof}
Let $A=[a_{ij}]\in\mathcal{PC}_{n}.$ For any $i,j\in\{1,\ldots,n\},$ with
$i>j,$ we have
\[
a_{ij}+a_{ji}=a_{ij}+\frac{1}{a_{ij}}\geq2,
\]
with equality if and only if $a_{ij}=a_{ji}=1.$ Since there are $\frac
{n^{2}-n}{2}$ such pairs $i,j$ and the diagonal entries of $A$ are $1,$ the
sum of the entries of $A\in\mathcal{PC}_{n}$ is at least $n+2\frac{n^{2}-n}%
{2}=n^{2},$ with equality if and only if all entries of $A$ are equal to $1.$
\end{proof}

\begin{remark}
\label{remlem}From Lemma \ref{ll1}, it follows that, if $r_{1}\geq\cdots\geq
r_{n}$ are the row sums of $A\in\mathcal{PC}_{n},$ then $r_{1}+\cdots
+r_{n}\geq n^{2}.$ This implies $r_{1}\geq n,$ with equality if and only if
$r_{1}=\cdots=r_{n}=n.$
\end{remark}

Since, by Lemma \ref{ldiag}, any $A\in\mathcal{PC}_{n}$ is similar to a
reciprocal matrix $A^{\prime}$ with constant row sums (equal to the Perron
eigenvalue of $A$)$,$ the well-known fact that the Perron eigenvalue of a
reciprocal matrix $A\in\mathcal{PC}_{n}$ is greater than or equal to $n$
follows from Remark \ref{remlem}. Moreover, the Perron eigenvalue is $n$ if
and only if all row sums of $A^{\prime}$ are $n,$ which implies $A^{\prime
}=\mathbf{J}_{n},$ by Lemma \ref{ll1}, that is, $A$ is consistent.

\begin{lemma}
\label{lrowcons}Let $A\in\mathcal{PC}_{n}$ be a consistent matrix. If $r$ is
the smallest row sum of $A$ then $r\leq n,$ with equality if and only if
$A=\mathbf{J}_{n}.$
\end{lemma}

\begin{proof}
Let $A=[a_{ij}]$ and $r_{i}$ be the sum of the entries in the $i$th row of
$A.$ Since $A$ is consistent, there are $w_{1},\ldots,w_{n}>0$ such that
$a_{ij}=\frac{w_{i}}{w_{j}}$ for all $i,j=1,\ldots,n.$ Let $l$ be such that
$\min_{i=1,\ldots,n}w_{i}=w_{l}.$ Then, the smallest row sum of $A$ is%
\[
r_{l}=\frac{w_{l}}{w_{1}}+\cdots+\frac{w_{l}}{w_{n}}\leq n,
\]
with equality if and only if all $w_{i}$'s are equal, that is, $A=\mathbf{J}%
_{n}.$
\end{proof}

\bigskip

Note that the claim in Lemma \ref{lrowcons} is not true for an arbitrary
$A\in\mathcal{PC}_{n}.$

\subsection{Results about efficiency\label{sdig}}

In \cite{blanq2006} the authors proved a useful result that gives a
characterization of efficiency in terms of a certain digraph. A shorter and
matricial proof of this result can be found in \cite{FJ3}.

\bigskip

Given $A\in\mathcal{PC}_{n}$ and $w=\left[
\begin{array}
[c]{ccc}%
w_{1} & \cdots & w_{n}%
\end{array}
\right]  ^{T}\in\mathcal{V}_{n}$, define $G(A,w)$ as the digraph with vertex
set $\{1,\ldots,n\}$ and a directed edge $i\rightarrow j,$ $i\neq j,$ if and
only if $\frac{w_{i}}{w_{j}}\geq a_{ij}$.

\begin{lemma}
\textrm{\cite{blanq2006}}\label{blanq} Let $A\in\mathcal{PC}_{n}$ and
$w\in\mathcal{V}_{n}.$ The vector $w$ is efficient for $A$ if and only if
$G(A,w)$ is a strongly connected digraph, that is, for all pairs of vertices
$i,j,$ with $i\neq j,$ there is a directed path from $i$ to $j$ in $G(A,w)$.
\end{lemma}

\bigskip We notice some relevant facts used later. \bigskip A digraph $G$ is
strongly connected if and only if its adjacency matrix is irreducible
\cite{HJ}. Thus, if $G(A,w)$ is not strongly connected, the matrix $\left[
\frac{w_{i}}{w_{j}}\right]  -A$ is permutationally similar to a matrix of the
form
\[
\left[
\begin{array}
[c]{cc}%
Q_{1} & >0\\
<0 & Q_{2}%
\end{array}
\right]  ,
\]
for some $Q_{1}\in M_{k}$ and $Q_{2}\in M_{n-k},$ with $1\leq k<n.$ Here $>0$
(resp. $<0)$ denotes a block of appropriate size with all entries positive
(resp. negative).

\bigskip

For $i\in\{1,\ldots,n\},$ $G(A(i),w(i))$ is the subgraph of $G(A,w)$ induced
by vertices $1,\ldots,i-1,i+1,\ldots,n.$ If $w(i)$ is efficient for $A(i)$ and
$w$ is inefficient for $A,$ then $G(A(i),w(i))$ is strongly connected and
$G(A,w)$ is not. Thus, vertex $i$ of $G(A,w)$ is either a sink vertex (that
is, a vertex with outdegree $0$) or a source vertex (that is, a vertex with
indegree $0$).

\bigskip

Next, we recall a result that allows us to simplify our proofs. It concerns
how $\mathcal{E}(A)$ changes when $A$ is subjected to a monomial similarity.

\begin{lemma}
\label{lsim}\textrm{\cite{Fu22}} Suppose that $A\in\mathcal{PC}_{n}$ and
$w\in\mathcal{E}(A).$ If $D\in M_{n}$ is a positive diagonal matrix ($P\in
M_{n}$ is a permutation matrix), then $Dw\in\mathcal{E}(DAD^{-1})$
($Pw\in\mathcal{E}(PAP^{T})$).
\end{lemma}

\bigskip We note that, if $D$ is a positive diagonal matrix, then the digraphs
$G(DAD^{-1},Dw)$ and $G(A,w)$ coincide.

\bigskip

In \cite{FJ1} the efficient extensions for $A\in\mathcal{PC}_{n}$ of an
efficient vector for a principal submatrix of $A$ in $\mathcal{PC}_{n-1}$ were
characterized. That is, for $i\in\{1,\ldots,n\},$ a necessary and sufficient
condition on the $i$th entry of $w\in\mathcal{V}_{n}$ was given for $w$ to be
efficient for $A$ when $w(i)$ is efficient for $A(i).$

In \textrm{\cite{FJ2}} it was shown that, if $w$ is efficient for
$A\in\mathcal{PC}_{n},$ $n>3,$ then there are two $(n-1)$-subvectors of $w$
efficient for the corresponding principal submatrices of $A.$ More formally,
there are $i,j\in\{1,\ldots,n\},$ with $i\neq j,$ such that $w(i)$ is
efficient for $A(i)$ and $w(j)$ is efficient for $A(j).$ We note here, for the
first time, that a converse of this result also holds. For $A\in
\mathcal{PC}_{n},$ with $n>2,$ and $w\in\mathcal{V}_{n},$ if there are two
$(n-1)$-subvectors of $w$ efficient for the corresponding principal
submatrices of $A$, then $w$ is efficient for $A.$

\begin{theorem}
\label{thind}Let $A\in\mathcal{PC}_{n},$ with $n>2$, and $w\in\mathcal{V}_{n}%
$. If there are $i,j\in\{1,\ldots,n\},$ with $i\neq j,$ such that $w(i)$ is
efficient for $A(i)$ and $w(j)$ is efficient for $A(j),$ then $w$ is efficient
for $A.$
\end{theorem}

\begin{proof}
The result is a consequence of Lemma \ref{blanq} and the fact that, if
$G(A(i),w(i))$ and $G(A(j),w(j)),$ $i\neq j,$ are strongly connected, then so
is $G(A,w).$ To see this, let $k,l\in\{1,\ldots,n\}$, $k\neq l.$ If
$j\notin\{k,l\},$ then there is a directed path from $k$ to $l$ in
$G(A(j),w(j))$. If $j\in\{k,l\}$ and $i\notin\{k,l\}$ then there is a directed
path from $k$ to $l$ in $G(A(i),w(i))$. If $k=i$ and $l=j,$ since $n>2,$ there
is a vertex $p\notin\{i,j\}$ in $G(A,w).$ Then, there is a directed path from
$k$ to $p$ in $G(A(j),w(j))$ and a directed path from $p$ to $l$ in
$G(A(i),w(i))$. In all cases, there is a directed path from $k$ to $l$ in
$G(A,w)$.
\end{proof}

\section{Further new lemmas and definitions\label{sectech}}

We start with an analytical lemma that will be crucial in deriving our main results.

\begin{lemma}
\label{lauxder}Let $a_{i}\geq0$ for $i=1,\ldots,k.$ For $x>0,$ define%
\[
f(x)=\frac{1}{x}+\frac{1}{a_{2}+x}+\cdots+\frac{1}{a_{k}+x}+1-a_{1}-x.
\]
Then, $f$ is a strictly decreasing continuous function with range
$\mathbb{R}.$ In particular, there is one and only one $x>0$ such that
$f(x)=0.$ If $a_{i}=0$ for $i=2,\ldots,k,$ $x$ is given by%
\begin{equation}
x=\frac{1-a_{1}+\sqrt{(1-a_{1})^{2}+4k}}{2}. \label{xx}%
\end{equation}

\end{lemma}

\begin{proof}
Clearly, $f$ is continuous and $f^{\prime}(x)<0$ for any $x>0,$ implying that
$f$ is a strictly decreasing function. Also,%
\[
\lim_{x\rightarrow+\infty}f(x)=-\infty\text{ and }\lim_{x\rightarrow0^{+}%
}f(x)=+\infty.
\]
Thus, the first claim follows which, obviously, implies the second one. The
last claim follows from a simple calculation.
\end{proof}

\bigskip

In the rest of this section we consider a matrix $B\in\mathcal{PC}_{k}$ and
denote by $r_{1}\geq\cdots\geq r_{k}$ the row sums of $B$ in nonincreasing
order. For $x>0,$ let $f$ be the function associated with $B$ defined by%
\begin{equation}
f(x)=\frac{1}{x}+\frac{1}{r_{1}-r_{2}+x}+\cdots+\frac{1}{r_{1}-r_{k}%
+x}+1-r_{1}-x. \label{ff1}%
\end{equation}

\begin{lemma}
\label{lf1}Let $B\in\mathcal{PC}_{k}.$ For $f$ as in (\ref{ff1}), we have
$f(1)\leq0,$ with equality if and only if $B=\mathbf{J}_{k}.$ In particular,
$f(x)=0$ implies $x\leq1.$
\end{lemma}

\begin{proof}
Since $r_{1}-r_{i}\geq0$ for $i=1,\ldots,k,$ we have%
\[
f(1)=1+\frac{1}{r_{1}-r_{2}+1}+\cdots+\frac{1}{r_{1}-r_{k}+1}-r_{1}\leq
k-r_{1}.
\]
From Lemma \ref{ll1} and Remark \ref{remlem}, we have $k-r_{1}\leq0$ with
equality if and only if $B=\mathbf{J}_{k},$ in which case $f(1)=0.$ The second
claim follows from the first one and Lemma \ref{lauxder}.
\end{proof}

\begin{definition}
We say that $B\in\mathcal{PC}_{k}$ is \emph{well-behaved of type I} if
$r_{1}-r_{k}\geq1,$ and is \emph{well-behaved of type II} if $r_{1}-r_{k}<1$
and
\[
f(1+r_{k}-r_{1})=\frac{1}{1+r_{k}-r_{1}}+\frac{1}{1+r_{k}-r_{2}}+\cdots
+\frac{1}{1+r_{k}-r_{k-1}}+1-r_{k}\geq0.
\]
We say that $B$ is \emph{well-behaved} if it is well-behaved of type I or of
type II.
\end{definition}

Not any matrix in $\mathcal{PC}_{k}$ is well-behaved. A matrix $B\in
\mathcal{PC}_{k}$ is not well-behaved if $r_{1}-r_{k}<1$ and $f(1+r_{k}%
-r_{1})<0.$ Note that the set of matrices $B\in\mathcal{PC}_{k}$ that are not
well-behaved is open.

\bigskip

The designation "well-behaved" follows from the fact that a reciprocal matrix
in $\mathcal{PC}_{n-1}$ that is not well-behaved cannot be extended to a
matrix in $\mathcal{PC}_{n}$ with efficient Perron vector $\mathbf{e}_{n},$ as
will be seen in Section \ref{secextin}.

\begin{example}
\label{ex1}The matrices%
\[
A_{1}=\left[
\begin{array}
[c]{ccc}%
1 & 2 & \frac{3}{5}\\
\smallskip\frac{1}{2} & 1 & 3\\
\frac{5}{3} & \frac{1}{3} & 1
\end{array}
\right]  \text{ and }A_{2}=\left[
\begin{array}
[c]{ccc}%
1 & \frac{6}{5} & 1\\
\frac{5}{6} & 1 & 1\\
1 & 1 & 1
\end{array}
\allowbreak\right]
\]
are well-behaved of types I and II, respectively. The matrices
\[
A_{3}=\left[
\begin{array}
[c]{ccc}%
1 & 5 & \frac{1}{5}\\
\frac{1}{5} & 1 & 5\\
5 & \frac{1}{5} & 1
\end{array}
\right]  \text{ and }A_{4}=\left[
\begin{array}
[c]{ccc}%
1 & \frac{1}{5} & \smallskip\frac{51}{10}\\
5 & 1 & \smallskip\frac{2}{9}\\
\frac{10}{51} & \frac{9}{2} & 1
\end{array}
\right]  ,
\]
(and any sufficiently small reciprocal perturbations of them), are not well-behaved.
\end{example}

Note that the notion of well-behaved is permutation similarity invariant.
Thus, typically, in the discussions to follow, when a row and a column are
added by bordering, it could be as well by inserting anywhere in the matrix.

\bigskip

A matrix in $\mathcal{PC}_{k}$ with constant row sums is not well-behaved of
type I$.$ Thus, from Lemma \ref{lf1}, we have the following.

\begin{lemma}
\label{lrsum}If $B\in\mathcal{PC}_{k}$ has constant row sums and is different
from $\boldsymbol{J}_{k}$ then $B$ is not well-behaved.
\end{lemma}

The next result will be a consequence of Theorems \ref{t2} and \ref{t6} given
later. However, we give here a direct proof of it.

\begin{proposition}
\label{lconswell}If $B\in\mathcal{PC}_{k}$ is consistent then $B$ is well-behaved.
\end{proposition}

\begin{proof}
It is enough to see that, if $r_{1}-r_{k}<1$ then $f(1+r_{k}-r_{1})\geq0,$ in
which $f$ is as in (\ref{ff1}). If $r_{1}-r_{k}<1,$ we have%
\[
f(1+r_{k}-r_{1})\geq\frac{k}{(r_{1}-r_{k})+(1+r_{k}-r_{1})}-r_{k}=k-r_{k}%
\geq0,
\]
where the last inequality follows from Lemma \ref{lrowcons}.
\end{proof}

\bigskip

There are consistent matrices well-behaved of type I and of type II.

\begin{example}
The consistent matrix $\boldsymbol{J}_{k}$ is well-behaved of type II while
the consistent matrix%
\[
\left[
\begin{array}
[c]{cc}%
1 & 5\mathbf{e}_{k-1}^{T}\\
\frac{1}{5}\mathbf{e}_{k-1} & \boldsymbol{J}_{k-1}%
\end{array}
\right]
\]
is well-behaved of type I.
\end{example}

\section{Extending a reciprocal matrix to one with a prescribed Perron
vector\label{sext}}

\begin{theorem}
\label{t1}For any $B\in\mathcal{PC}_{n-1}$ and $w\in\mathcal{V}_{n},$ there is
$A\in\mathcal{PC}_{n}$ with Perron vector $w$ and such that $A(n)=B.$
Moreover, $A$ is unique.
\end{theorem}

\begin{proof}
Taking into account Lemmas \ref{ldiagon} and \ref{ldiag}, we assume
$w=\mathbf{e}_{n}.$ Moreover, still by Lemma \ref{ldiagon}, we assume
$r_{1}\geq\cdots\geq r_{n-1},$ in which $r_{i},$ $i=1,\ldots,n-1,$ is the sum
of the entries in the $i$th row of $B.$

Since we want $A$ with constant row sums and such that $A(n)=B$, the last
column of $A$ should be of the form
\begin{equation}
\left[
\begin{array}
[c]{ccccc}%
x & r_{1}-r_{2}+x & \cdots & r_{1}-r_{n-1}+x & 1
\end{array}
\right]  ^{T}, \label{col}%
\end{equation}
for some $x>0$. Since $A$ is reciprocal and the sum of the entries in the last
row of $A$ should be $r_{1}+x,$ we also have%
\begin{equation}
\frac{1}{x}+\frac{1}{r_{1}-r_{2}+x}+\cdots+\frac{1}{r_{1}-r_{n-1}+x}%
+1=r_{1}+x. \label{eq1}%
\end{equation}
By Lemma \ref{lauxder}, the existence and uniqueness of $x>0$ satisfying
(\ref{eq1}) follows. For such $x$, $A$ is as desired.
\end{proof}

\bigskip

By iterating the application of Theorem \ref{t1}, we get the following.

\begin{corollary}
Let $B\in\mathcal{PC}_{k},$ $k<n,$ and $w^{(k+i)}\in\mathcal{V}_{k+i},$
$i=1,\ldots,n-k.$ Then, there is $A\in\mathcal{PC}_{n}$ such that
$A[\{1,\ldots,k\}]=B$ and $A[\{1,\ldots,k+i\}]$ has Perron vector $w^{(k+i)},$
$i=1,\ldots,n-k.$ Moreover, $A$ is unique.
\end{corollary}

We also have the following consequence of Theorem \ref{t1}.

\begin{corollary}
\label{t3}Let $A_{1},A_{2}\in\mathcal{PC}_{n}$ with $A_{1}(n)=A_{2}(n).$ If
$w^{(1)}$ and $w^{(2)}$ are the Perron vectors of $A_{1}$ and $A_{2},$
respectively, and $w^{(1)}(n)$ and $w^{(2)}(n)$ are proportional$,$ then
$A_{2}=DA_{1}D^{-1}$ for some $D=I_{n-1}\oplus\lbrack c],$ with $c>0.$
\end{corollary}

\begin{proof}
Using Lemma \ref{ldiag}, we can conclude that there is a (unique) positive
diagonal matrix $D$ such that the Perron vector $Dw^{(1)}$ of $DA_{1}D^{-1}$
is a multiple of $w^{(2)}.$ Moreover, since $w^{(1)}(n)$ and $w^{(2)}(n)$ are
multiples, the first $n-1$ entries of $D$ are constant. Thus, $DA_{1}D^{-1}$
satisfies $(DA_{1}D^{-1})(n)=A_{2}(n)$ and has Perron vector $w^{(2)}.$ By the
uniqueness statement in Theorem \ref{t1}, we have $DA_{1}D^{-1}=A_{2},$
proving the claim.
\end{proof}

\section{Matrices with inefficient Perron vector\label{s5}}

In this section we construct matrices $A\in\mathcal{PC}_{n}$ with inefficient
Perron vector and a prescribed principal submatrix. We start by studying the
possibility of extending a matrix in $\mathcal{PC}_{n-1}$ to one in
$\mathcal{PC}_{n}$ with inefficient Perron vector $\mathbf{e}_{n}.$ Based on
the result obtained, we then extend any reciprocal matrix to one of arbitrary
larger size $n$ and with inefficient Perron vector, except if the fixed matrix
is consistent of size $n-1,$ in which case we show that such an extension does
not exist.

\subsection{Extending a not well-behaved reciprocal matrix to one with
inefficient Perron vector $\mathbf{e}_{n}$\label{secextin}}

Let $A\in\mathcal{PC}_{n}$ and suppose that $\mathbf{e}_{n}$ is the Perron
vector of $A$. We next give sufficient conditions for the Perron vector of $A$
to be inefficient. Moreover, if $\mathbf{e}_{n-1}$ is efficient for $A(n),$
these conditions are also necessary.

\begin{theorem}
\label{t2}Let $A\in\mathcal{PC}_{n}$ and suppose that $\mathbf{e}_{n}$ is the
Perron vector of $A$. If $A(n)$ is not well-behaved, then the Perron vector of
$A$ is inefficient for $A$ and vertex $n$ is a sink vertex of $G(A,\mathbf{e}%
_{n})$. As for a converse, if $\mathbf{e}_{n-1}$ is efficient for $A(n)$ and
the Perron vector of $A$ is inefficient for $A,$ then $A(n)$ is not well-behaved.
\end{theorem}

\begin{proof}
Let $r_{i}$ be the $i$th row sum of $A(n),$ $i=1,\ldots,n-1.$ Taking into
account Lemmas \ref{ldiagon} and \ref{lsim}, we may assume $r_{1}\geq
\cdots\geq r_{n-1}.$ Using arguments similar to those in the proof of Theorem
\ref{t1}, the last column of $A$ with Perron vector $\mathbf{e}_{n}$ is as in
(\ref{col}), with $x>0$ satisfying (\ref{eq1}).

If $A(n)$ is not well-behaved then%
\[
r_{1}-r_{n-1}<1\text{ and }f(1+r_{n-1}-r_{1})<0,
\]
implying, taking into account Lemma \ref{lauxder},
\begin{equation}
x<1+r_{n-1}-r_{1}\leq1. \label{f3}%
\end{equation}
Since (\ref{f3}) implies $r_{1}-r_{n-1}+x<1,$ and, on the other hand,
\[
x\leq r_{1}-r_{2}+x\leq\cdots\leq r_{1}-r_{n-1}+x,
\]
we have that all the off-diagonal entries of the last column of $A$ are
smaller than $1,$ implying that $G(A,\mathbf{e}_{n})$ is not strongly
connected and, thus, by Lemma \ref{blanq}, $\mathbf{e}_{n}$ is inefficient for
$A.$ Note that vertex $n$ is a sink. This proves the first claim.

As for the second claim, if $\mathbf{e}_{n-1}$ is efficient for $A(n),$ then,
by Lemma \ref{blanq}, $G(A(n),\mathbf{e}_{n-1})$ is strongly connected. If
$\mathbf{e}_{n}$ is inefficient for $A,$ then $G(A,\mathbf{e}_{n})$ is not
strongly connected. Thus, the off-diagonal entries of the last column of $A$
are all greater than $1$ or all less than $1.$ Since (\ref{eq1}) holds, by
Lemma \ref{lf1}, $x\leq1.$ Hence, the off-diagonal entries of the last column
of $A$ are all less than $1.$ Then, $r_{1}-r_{n-1}+x<1,$ that is,
$r_{1}-r_{n-1}<1$ and $x<1+r_{n-1}-r_{1}.$ By Lemma \ref{lauxder}, for $f$ as
in (\ref{ff1}) with $k=n-1,$ we have $f(1+r_{n-1}-r_{1})<0,$ as $f(x)=0.$ This
implies that $A(n)$ is not well-behaved, proving the result.
\end{proof}

\bigskip Note that, by Theorem \ref{thind}, if $\mathbf{e}_{n}$ is inefficient
for $A,$ then $\mathbf{e}_{n-1}$ is inefficient for $A(i),$ for at least $n-1$
distinct $i$'s$.$

\bigskip

By Theorem \ref{t1}, any matrix $B\in\mathcal{PC}_{n-1}$ can be extended to a
matrix in $\mathcal{PC}_{n}$ with Perron vector $\mathbf{e}_{n}.$ Following
the idea in the proof of that result, in the next example we construct a
matrix $A\in\mathcal{PC}_{4}$ with constant row sums and a prescribed, not
well-behaved, principal submatrix in $\mathcal{PC}_{3}.$ According to Theorem
\ref{t2}, the Perron vector of $A$, $\mathbf{e}_{4}$, is inefficient for $A.$

\begin{example}
\label{exx0}We determine $A\in\mathcal{PC}_{4}$ with Perron vector
$\mathbf{e}_{4}$ (that is, with constant row sums) and such that
\begin{equation}
A(4)=\left[
\begin{array}
[c]{ccc}%
1 & \frac{1}{5} & \smallskip\frac{51}{10}\\
5 & 1 & \smallskip\frac{2}{9}\\
\frac{10}{51} & \frac{9}{2} & 1
\end{array}
\right]  . \label{A4}%
\end{equation}
Note that $A(4)$ is not well-behaved, as observed in Example \ref{ex1}, and
has nonincreasing row sums. The matrix $A$ should be of the form
\[
A=\left[
\begin{array}
[c]{cccc}%
1 & \frac{1}{5} & \smallskip\frac{51}{10} & x\\
5 & 1 & \smallskip\frac{2}{9} & x+\frac{1}{5}+\frac{51}{10}-5-\frac{2}{9}\\
\frac{10}{51} & \frac{9}{2} & 1 & \smallskip x+\frac{1}{5}+\frac{51}{10}%
-\frac{9}{2}-\frac{10}{51}\\
\frac{1}{x} & \smallskip\frac{1}{x+\frac{1}{5}+\frac{51}{10}-5-\frac{2}{9}} &
\frac{1}{x+\frac{1}{5}+\frac{51}{10}-\frac{9}{2}-\frac{10}{51}} & 1
\end{array}
\right]  ,
\]
with $x$ satisfying%
\[
\frac{1}{x}+\frac{1}{x+\frac{1}{5}+\frac{51}{10}-5-\frac{2}{9}}+\frac
{1}{x+\frac{1}{5}+\frac{51}{10}-\frac{9}{2}-\frac{10}{51}}+1=1+\frac{1}%
{5}+\frac{51}{10}+x.
\]
A calculation gives $x=0.39137,$ implying that
\[
\allowbreak A=\allowbreak\left[
\begin{array}
[c]{cccc}%
1 & \frac{1}{5} & \smallskip\frac{51}{10} & 0.39137\\
5 & 1 & \smallskip\frac{2}{9} & 0.46915\\
\frac{10}{51} & \frac{9}{2} & 1 & 0.99529\\
2.5551 & 2.1315 & 1.0047 & 1
\end{array}
\right]  \allowbreak.\allowbreak
\]
By Theorem \ref{t2}, $A$ has inefficient Perron vector $\mathbf{e}_{4}$ ($4$
is a sink vertex of $G(A,\mathbf{e}_{4})$).
\end{example}

\bigskip Theorem \ref{t2} leaves the question of characterizing the matrices
$A\in\mathcal{PC}_{n}$ with inefficient Perron vector $\mathbf{e}_{n}$ and
such that $A(i)$ is well-behaved for all $i\in\{1,\ldots,n\}.$ In this case,
by the theorem, $\mathbf{e}_{n-1}$ is inefficient for $A(i),$ $i=1,\ldots,n.$
Moreover, by Lemma \ref{lrsum}, $\mathbf{e}_{n-1}$ is not the Perron vector of
$A(i)$. Note that, since $\mathbf{e}_{n-1}$ is inefficient for $A(i),$ the
matrices $A(i)$ and $\mathbf{J}_{n-1}$ are distinct. Next we give an example
of such a matrix $A$. We show in Section \ref{sec4} that we should have $n>4.$

\begin{example}
\label{ex4}Let
\[
A=\left[
\begin{array}
[c]{cccccc}%
1 & 1.2783 & 0.2364 & 1.0245 & 2.0221 & 4.5197\\
0.7823 & 1 & 2.4655 & 1.6028 & 2.1091 & 2.1214\\
4.2304 & 0.4056 & 1 & 1.3002 & 1.7109 & 1.4340\\
0.9761 & 0.6239 & 0.7691 & 1 & 6.5795 & 0.1324\\
0.4945 & 0.4741 & 0.5845 & 0.1520 & 1 & 7.3759\\
0.2213 & 0.4714 & 0.6973 & 7.5555 & 0.1356 & 1
\end{array}
\right]  \in\mathcal{PC}_{6}.
\]
The matrix $A$ has Perron eigenvector $\mathbf{e}_{6},$ which is inefficient
for $A.$ It can be seen that, for any $i\in\{1,\ldots,6\},$ $\mathbf{e}_{5}$
is inefficient for $A(i)$ and $A(i)$ is well-behaved ($A(2)$ is well-behaved
of type II and $A(i)$ is well-behaved of type I for $i\neq2$). Observe that
the digraph $G(A,w)$ has no sink nor source vertex, as each row of $A$ has
entries greater than $1$ and entries less than $1.$
\end{example}

As a consequence of Theorem \ref{t2}, we have the following important result
that allows us to easily construct reciprocal matrices with inefficient Perron vector.

\begin{corollary}
\label{c4}Let $A\in\mathcal{PC}_{n}$ and $w\in\mathcal{V}_{n}$ be the Perron
vector of $A.$ Suppose that $A(n)$ is inconsistent. If $w(n)$ is the Perron
vector of $A(n)$ then $w$ is inefficient for $A$ and $n$ is a sink vertex of
$G(A,w)$.
\end{corollary}

\begin{proof}
By Lemmas \ref{ldiagon}, \ref{ldiag} and \ref{lsim}, and since, for a positive
diagonal matrix $D\in M_{n},$ $G(DAD^{-1},Dw)$ and $G(A,w)$ coincide, we may
assume $w=\mathbf{e}_{n}.$ Since $A(n)$ is inconsistent with constant row
sums, by Lemma \ref{lrsum}, it is not well-behaved. Thus, the result follows
from Theorem \ref{t2}.
\end{proof}

\bigskip Observe that, if $w=\left[
\begin{array}
[c]{ccc}%
w_{1} & \cdots & w_{n}%
\end{array}
\right]  ^{T}$ is the Perron vector of $A$ and $A(n)$ is consistent with
Perron vector $w(n),$ then, by the uniqueness statement in Theorem \ref{t1},
$A=\left[  \frac{w_{i}}{w_{j}}\right]  $ ($A$ is consistent) and $w$ is
efficient for $A$ (as will follow from Theorem \ref{t6}). Thus, the assumption
in Corollary \ref{c4} that $A(n)$ is inconsistent is necessary.

\bigskip

In Corollary \ref{c4} we gave sufficient conditions for $G(A,w)$ to have a
sink vertex, in which $w$ is the Perron vector of $A.$ Another sufficient
condition is given next.

\begin{corollary}
\label{ceff}Let $A\in\mathcal{PC}_{n}$ and $w$ be the Perron vector of $A.$ If
$w$ is inefficient for $A$ and $w(n)$ is efficient for $A(n)$, then $n$ is a
sink vertex of $G(A,w)$.
\end{corollary}

\begin{proof}
By Lemmas \ref{ldiagon}, \ref{ldiag} and \ref{lsim}, we assume that
$w=\mathbf{e}_{n}.$ If $\mathbf{e}_{n}$ is inefficient for $A$ and
$\mathbf{e}_{n-1}$ is efficient for $A(n),$ by Theorem \ref{t2}, $A(n)$ is not
well-behaved, which implies, by the same theorem, that $n$ is a sink vertex of
$G(A,w).$
\end{proof}

\bigskip

From Corollary \ref{c4}, we can easily construct a reciprocal matrix with
inefficient Perron vector by adding a column with constant off-diagonal
entries (and the corresponding reciprocal row) to an arbitrary inconsistent
reciprocal matrix with constant row sums. More precisely, if $T\in
\mathcal{PC}_{n-1}$ has constant row sums, say equal to $r,$ and is
inconsistent, by Corollary \ref{c4}, any matrix $A$ of the form%
\begin{equation}
A=\left[
\begin{array}
[c]{cc}%
T & a\mathbf{e}_{n-1}\\
\frac{1}{a}\mathbf{e}_{n-1}^{T} & 1
\end{array}
\right]  , \label{part}%
\end{equation}
with $a>0,$ has inefficient Perron vector. Note that, for $D^{-1}%
=I_{n-1}\oplus\left[  x_{0}\right]  ,$ in which $x=ax_{0}$ satisfies
(\ref{xx}) (with $a_{1}=r$ and $k=n-1$), the matrix $DAD^{-1}$ has constant
row sums. Thus, the Perron vector of $A\ $is $\left[
\begin{array}
[c]{cc}%
\mathbf{e}_{n-1}^{T} & x_{0}%
\end{array}
\right]  ^{T}.$

According to Lemma \ref{lsim}, any matrix monomial similar to a matrix as in
(\ref{part}), with $T$ inconsistent with constant row sums, has inefficient
Perron vector.

\bigskip

We now focus on the construction of matrices $T$ as above. First, note that,
from Lemma \ref{ldiag}, given any $B\in\mathcal{PC}_{k}$ with Perron vector
$v,$ for $D^{-1}=\operatorname*{diag}(v),$ the matrix $T=DBD^{-1}$ has
constant row sums. Moreover, if $B$ is inconsistent, then so is $T.$

Next we present structured classes of inconsistent reciprocal matrices $T$
with constant row sums. One such class consists of matrices of the form
\begin{equation}
T=\left[
\begin{array}
[c]{ccccccc}%
1 & b & 1 & \cdots & \cdots & 1 & \frac{1}{b}\\
\frac{1}{b} & 1 & b & 1 & \cdots & 1 & 1\\
1 & \frac{1}{b} & 1 & b & \ddots &  & 1\\
\vdots & 1 & \frac{1}{b} & 1 & \ddots & \ddots & \\
\vdots & \vdots & \ddots & \ddots & \ddots & b & 1\\
1 & 1 &  & \ddots & \ddots & 1 & b\\
b & 1 & 1 &  &  & \frac{1}{b} & 1
\end{array}
\right]  \in\mathcal{PC}_{k}, \label{Boz}%
\end{equation}
in which $k\geq3$ and $b$ is an arbitrary positive number different from $1.$
The matrices of the form (\ref{part}) with $T$ as in (\ref{Boz}) $(b\neq1)$
form the class of reciprocal matrices presented in \cite{bozoki2014} for which
inefficiency of the Perron vector was noted. Thus, the class of
\cite{bozoki2014} is part of our much more general construction.

Another class of inconsistent reciprocal matrices with constant row sums
consists of the Toeplitz matrices $T\in\mathcal{PC}_{k},$ with $k\geq3$ odd$,$
whose first row is%
\[
\left[
\begin{array}
[c]{cccccccc}%
1 & b & \frac{1}{b} & b & \frac{1}{b} & \cdots & b & \frac{1}{b}%
\end{array}
\right]  ,
\]
with $b\neq1.$ For example, for $k=5,$
\[
T=\left[
\begin{array}
[c]{ccccc}%
1 & b & \frac{1}{b} & b & \frac{1}{b}\\
\frac{1}{b} & 1 & b & \frac{1}{b} & b\\
b & \frac{1}{b} & 1 & b & \frac{1}{b}\\
\frac{1}{b} & b & \frac{1}{b} & 1 & b\\
b & \frac{1}{b} & b & \frac{1}{b} & 1
\end{array}
\right]  .
\]
Finally, we note that, if $T_{0},T_{1}\in\mathcal{PC}_{k}$ have constant row
sums and at least one is different from $\mathbf{J}_{k}$, then%
\[
\left[
\begin{array}
[c]{cc}%
T_{0} & T_{1}\\
T_{1} & T_{0}%
\end{array}
\right]  \in\mathcal{PC}_{2k}\text{ }%
\]
has constant row sums and is inconsistent. The same holds for%
\[
\left[
\begin{array}
[c]{cc}%
T_{0} & x\mathbf{e}_{k}\\
\frac{1}{x}\mathbf{e}_{k}^{T} & 1
\end{array}
\right]  \in\mathcal{PC}_{k+1}%
\]
in which $x$ is given by (\ref{xx}), with $a_{1}$ being the constant row sums
of $T_{0}$ (assumed inconsistent)$.$ Thus, we can construct inductively
infinitely many inconsistent reciprocal matrices of any size $n>3$ with
constant row sums.

\bigskip We conclude this section by showing that, if $A\in\mathcal{PC}_{n}$
has a consistent $(n-1)$-by-$(n-1)$ principal submatrix, then the Perron
vector of $A\ $is efficient for $A.$ This means that inserting a row (and
corresponding reciprocal column) to a consistent matrix gives a new reciprocal
matrix with efficient Perron vector. This fact generalizes the result of
\cite{p6} and a special case in \cite{p2}, and gives a complete picture that
includes each.

To give the theorem, we need the following lemma.

\begin{lemma}
\label{l22}Let $r_{1},\ldots,r_{k}$ be positive numbers such that $r_{1}\geq
k,$ $r_{k}\leq k$ and $r_{1}\geq r_{2}\geq\cdots\geq r_{k}.$ If $x_{1}%
,\ldots,x_{k}$ are positive numbers such that%
\[
r_{1}+x_{1}=r_{2}+x_{2}=\cdots=r_{k}+x_{k}=1+\frac{1}{x_{1}}+\cdots+\frac
{1}{x_{k}},
\]
then $x_{1}\leq1$ and $x_{k}\geq1.$
\end{lemma}

\begin{proof}
Let $r=r_{1}+x_{1}$ and suppose that the hypotheses in the statement hold.
Suppose that $x_{1}>1.$ Then, $r>k+1$ and $x_{i}>1$ for $i=2,\ldots,k.$ The
latter implies that%
\[
r=1+\frac{1}{x_{1}}+\cdots+\frac{1}{x_{k}}<k+1,
\]
a contradiction$.$

Now suppose that $x_{k}<1.$ Then, $x_{i}<1$ for $i=1,\ldots,k-1.$ This implies
that%
\[
r=1+\frac{1}{x_{1}}+\cdots+\frac{1}{x_{k}}>k+1,
\]
a contradiction, since $r=r_{k}+x_{k}<k+1.$
\end{proof}

\begin{theorem}
\label{t6}Let $A\in\mathcal{PC}_{n}.$ Suppose that $A(n)$ is consistent. Then
the Perron vector of $A$ is efficient for $A$.
\end{theorem}

\begin{proof}
Taking into account Lemmas \ref{ldiagon}, \ref{ldiag} and \ref{lsim}, we may
assume that the Perron vector of $A$ is $\mathbf{e}_{n}.$ Moreover, with a
possible permutation similarity on $A(n)$, we may assume $r_{1}\geq\cdots\geq
r_{n-1},$ in which $r_{i}$ is the sum of the entries in the $i$th row of
$A(n).$ Since $A(n)$ is consistent, we have
\[
A(n)=\left[
\begin{array}
[c]{c}%
w_{1}\\
\vdots\\
w_{n-1}%
\end{array}
\right]  \left[
\begin{array}
[c]{ccc}%
\frac{1}{w_{1}} & \cdots & \frac{1}{w_{n-1}}%
\end{array}
\right]  ,
\]
for some positive numbers $w_{i},$ $i=1,\ldots,n-1.$ As $A(n)$ has
nonincreasing row sums, we have $w_{1}\geq\cdots\geq w_{n-1}.$ Then, the upper
diagonal entries of $A(n)$ are greater than or equal to $1,$ implying that
$G(A(n),\mathbf{e}_{n-1})$ contains the path $(n-1)\rightarrow\cdots
\rightarrow2\rightarrow1.$ Thus, to show that $G(A,\mathbf{e}_{n})$ is
strongly connected, it is enough to see that $1\rightarrow n$ and
$n\rightarrow(n-1)$ are edges in $G(A,\mathbf{e}_{n}),$ that is, the entry of
$A$ in position $1,n$ is less than or equal to $1$ and the one in position
$n-1,n$ is greater than or equal to $1.$ By Lemma \ref{lrowcons} and Remark
\ref{remlem}, we have $r_{1}\geq n-1$ and $r_{n-1}\leq n-1.$ Since the row
sums of $A$ are constant, the claim about the entries of $A$ in positions
$1,n$ and $n-1,n$ follows from Lemma \ref{l22}.
\end{proof}

\bigskip

In this section we focused on the extension of a not well-behaved reciprocal
matrix (in particular, an inconsistent reciprocal matrix with constant row
sums) to a reciprocal matrix with inefficient Perron vector $\mathbf{e}_{n}.$
Using these ideas, in the next section, we summarize how to construct matrices
in $\mathcal{PC}_{n}$ with inefficient Perron vector and a prescribed
principal submatrix in $\mathcal{PC}_{k},$ $k<n$.

\subsection{An algorithm to extend an arbitrary reciprocal matrix to one with
an inefficient Perron vector\label{salg}}

Here we give an algorithm to construct a matrix $A\in\mathcal{PC}_{n}$ with
prescribed principal submatrix $B\in\mathcal{PC}_{k},$ $k<n,$ and with
inefficient Perron vector. We assume that $B$ is inconsistent if $k=n-1$, as
otherwise such a construction does not exist by Theorem \ref{t6} (however, if
$B$ is consistent and $k\leq n-2,$ such an $A$ does exist).

\bigskip\textbf{Algorithm }Let\textbf{ }$B\in\mathcal{PC}_{k}$ be given ($B$
is inconsistent if $k=n-1$).

\bigskip

let $S\in\mathcal{PC}_{n-1}$ be an inconsistent matrix such that
$S[\{1,\ldots,k\}]=B$

($S=B$ if $k=n-1$)

let $v$ be the Perron vector of $S$

let $R^{-1}=\operatorname*{diag}(v)$

let $C=RSR^{-1}$

let $a>0$ be arbitrary

let $A^{\prime}=\left[
\begin{array}
[c]{cc}%
C & a\mathbf{e}_{n-1}\\
\frac{1}{a}\mathbf{e}_{n-1}^{T} & 1
\end{array}
\right]  ,$

\smallskip let $D=R\oplus\lbrack c]$ for some $c>0$

let $A=D^{-1}A^{\prime}D$

\bigskip

From the discussion in Section \ref{secextin}, the Perron vector of the matrix
$A^{\prime}$ in the algorithm is inefficient for $A^{\prime}.$ Note that the
matrix $C$ has constant row sums. By Lemma \ref{lsim}, the Perron vector of
the matrix $A$ obtained by the algorithm is inefficient for $A.$ Moreover,
$A(n)=S.$

\bigskip

If $k<n-1,$ the matrix $S$ in the algorithm can be any inconsistent extension
of $B\in\mathcal{PC}_{k}$. For example, it can be constructed in $n-k-1$
steps, giving matrices $S_{1}\in\mathcal{PC}_{k+1},\ldots,S_{n-k-1}%
\in\mathcal{PC}_{n-1},$ with $S_{1}[\{1,\ldots,k\}]=B$ and $S_{i}%
[\{1,\ldots,k+i-1\}]=S_{i-1}.$ As long as $S=S_{n-k-1}$ is inconsistent, each
$S_{i}$ can have any desired Perron vector, by applying Theorem \ref{t1}, or
can be constructed to have inefficient Perron vector by applying the algorithm
(if $S_{i-1}$ is not consistent), or to have efficient Perron vector (by
applying Theorem \ref{t7} in Section \ref{s6}). Of course, different choices
of $S$ produce different matrices $A.$

\begin{example}
\label{exsubvin}We construct a matrix $A\in\mathcal{PC}_{5}$ with inefficient
Perron vector and such that
\[
A([1,2,3])=\left[
\begin{array}
[c]{ccc}%
1 & 2 & \frac{3}{5}\\
\smallskip\frac{1}{2} & 1 & 3\\
\frac{5}{3} & \frac{1}{3} & 1
\end{array}
\right]  .
\]
Let
\[
S=\left[
\begin{array}
[c]{cccc}%
1 & 2 & \smallskip\frac{3}{5} & 2\\
\smallskip\frac{1}{2} & 1 & 3 & \frac{1}{2}\\
\smallskip\frac{5}{3} & \frac{1}{3} & 1 & \frac{3}{2}\\
\smallskip\frac{1}{2} & 2 & \frac{2}{3} & 1
\end{array}
\right]
\]
be an extension in $\mathcal{PC}_{4}$ of the previous matrix in $\mathcal{PC}%
_{3}.$ The Perron vector of $S$ is$\allowbreak$ $\allowbreak$%
\[
v=\left[
\begin{array}
[c]{cccc}%
1.3348 & 1.1829 & 1.0946 & 1
\end{array}
\allowbreak\right]  ^{T}.
\]
Let $R^{-1}=\operatorname*{diag}(v).$ Then, $\allowbreak$
\[
C=RSR^{-1}=\left[
\begin{array}
[c]{cccc}%
1 & 1.7724 & 0.4920 & 1.4984\\
0.5642 & 1 & 2.7761 & 0.4227\\
2.0324 & 0.3602 & 1 & 1.3704\\
0.6674 & 2.3658 & 0.7297 & 1
\end{array}
\right]  \allowbreak.
\]
For any $a>0,$ the matrix$\allowbreak$%
\[
A^{\prime}=\left[
\begin{array}
[c]{ccccc}%
1 & 1.7724 & 0.4920 & 1.4984 & a\\
0.5642 & 1 & 2.7761 & 0.4227 & a\\
2.0324 & 0.3602 & 1 & 1.3704 & a\\
0.6674 & 2.3658 & 0.7297 & 1 & a\\
\frac{1}{a} & \frac{1}{a} & \frac{1}{a} & \frac{1}{a} & 1
\end{array}
\right]  ,
\]
has inefficient Perron vector. Taking $a=1$ and $D=R\oplus\lbrack1],\ $we get
\[
A=D^{-1}A^{\prime}D=\left[
\begin{array}
[c]{ccccc}%
1 & 2 & \smallskip\frac{3}{5} & 2 & 1.3348\\
\smallskip\frac{1}{2} & 1 & 3 & \frac{1}{2} & 1.1829\\
\smallskip\frac{5}{3} & \frac{1}{3} & 1 & \frac{3}{2} & 1.0946\\
\smallskip\frac{1}{2} & 2 & \frac{2}{3} & 1 & 1\\
0.7492 & 0.8454 & 0.9136 & 1 & 1
\end{array}
\right]  .
\]
Then, $A([1,2,3])=B$ (in fact, $A([1,2,3,4])=S$) and the Perron vector $w$ of
$A$ is inefficient for $A.$ Moreover, vertex $5$ is a sink vertex of $G(A,w).$
\end{example}

In the example, we could also have started with a $3$-by-$3$ consistent matrix
and have terminated with a $5$-by-$5$ reciprocal matrix whose Perron vector
was inefficient.

\bigskip

Finally, we observe that the Perron vector is a continuous function of the
entries of a matrix. So, if the Perron vector of $A\in\mathcal{PC}_{n}$ is
inefficient for $A$, then the Perron vector of any sufficiently small
reciprocal perturbation of $A$ is also inefficient for $A$ and for the
perturbation of $A$ (see \cite{FJ1}).

\subsection{The $4$-by-$4$ reciprocal matrices with inefficient Perron
vector\label{sec4}}

In this section we give a characterization of all matrices in $\mathcal{PC}%
_{4}$ whose Perron vector is inefficient. Recall that the Perron vector of any
matrix in $\mathcal{PC}_{n},$ with $n\leq3$, is efficient. The case $n=2$ is
trivial because the matrix is consistent. The case $n=3$ is covered by
\cite{p6, CFF}.

When the Perron vector $w$ is inefficient for $A\in\mathcal{PC}_{n},$ it may
happen that no $(n-1)$-subvector of $w$ is efficient for the corresponding
principal submatrix of $A,$ as illustrated next for $n=5$ (see also Example
\ref{ex4} for $n=6$). However, this does not happen when $n=4,$ which is an
important fact in obtaining the characterization mentioned above.

\begin{example}
Consider%
\[
A=\left[
\begin{array}
[c]{ccccc}%
1 & 2.032 & 0.53386 & 0.86855 & 0.88385\\
0.492\,3 & 1 & 2.1018 & 0.88907 & 0.83513\\
1.8731 & 0.475\,78 & 1 & 0.97616 & 0.99334\\
1.1513 & 1.\,\allowbreak124\,8 & 1.0244 & 1 & 1.\,\allowbreak0176\\
1.1314 & 1.\,\allowbreak197\,4 & 1.0067 & 0.9827 & 1
\end{array}
\right]  \in\mathcal{PC}_{5}.
\]
The vector $\mathbf{e}_{5}$ is the Perron vector of $A$ and is inefficient for
$A.$ Moreover, for any $i\in\{1,2,3,4,5\},$ $\mathbf{e}_{4}$ is inefficient
for $A(i).$ We note that vertex $4$ is a sink vertex of $G(A,\mathbf{e}_{5}).$
\end{example}

\begin{theorem}
\label{t5}Let $A\in\mathcal{PC}_{4}.$ If $w$ is the Perron vector of $A,$ then
there is an $i\in\{1,2,3,4\}$ such that $w(i)$ is efficient for $A(i).$
\end{theorem}

\begin{proof}
By Lemmas \ref{ldiagon}, \ref{ldiag} and \ref{lsim}, we may assume that
$w=\mathbf{e}_{4}.$ Suppose that $\mathbf{e}_{3}$ is not efficient for $A(4).$
Then, by Lemma \ref{blanq}, $G(A(4),\mathbf{e}_{3})$ is not strongly
connected. Hence, $A(4)$ is permutationally similar to a matrix of one of the
following forms%
\begin{equation}
\left[
\begin{tabular}
[c]{c|cc}%
$1$ & $<1$ & $<1$\\\hline
$>1$ & $1$ & $\geq1$\\
$>1$ & $\leq1$ & $1$%
\end{tabular}
\ \right]  \text{ or }\left[
\begin{tabular}
[c]{cc|c}%
$1$ & $\geq1$ & $<1$\\
$\leq1$ & $1$ & $<1$\\\hline
$>1$ & $>1$ & $1$%
\end{tabular}
\ \right]  . \label{k}%
\end{equation}
By $>1,$ $<1,$ $\geq1$ and $\leq1,$ we denote an entry greater than $1,$ less
than $1,$ greater than or equal to $1$ and less than or equal to $1,$
respectively. By Lemmas \ref{ldiagon} and \ref{lsim}, we may assume that
$A(4)$ has one of the forms in (\ref{k}). Suppose that $A(4)$ is as the first
matrix in (\ref{k}). The proof of the other case is similar. Since $A$ has
constant row sums, and these sums are at least $4$ (by Remark \ref{remlem}),
then the entry in position $1,4$ of $A$ is greater than $1,$ and thus $A$ has
the form
\[
\left[
\begin{array}
[c]{cccc}%
1 & <1 & <1 & >1\\
>1 & 1 & \geq1 & \\
>1 & \leq1 & 1 & \\
<1 &  &  & 1
\end{array}
\right]  .
\]
Then, using similar arguments, in positions $4,2$ or $4,3$ the matrix $A$ has
an entry greater than $1.$ In the first case $A$ has the form
\[
\left[
\begin{array}
[c]{cccc}%
1 & <1 & <1 & >1\\
>1 & 1 & \geq1 & <1\\
>1 & \leq1 & 1 & \\
<1 & >1 &  & 1
\end{array}
\right]  ,
\]
implying that $1\rightarrow2\rightarrow4\rightarrow1$ is a cycle in
$G(A(3),\mathbf{e}_{3})$ and, thus, $\mathbf{e}_{3}$ is efficient for $A(3).$
In the second case, $A$ has the form%
\[
\left[
\begin{array}
[c]{cccc}%
1 & <1 & <1 & >1\\
>1 & 1 & \geq1 & \\
>1 & \leq1 & 1 & <1\\
<1 &  & >1 & 1
\end{array}
\right]  ,
\]
implying that $1\rightarrow3\rightarrow4\rightarrow1$ is a cycle in
$G(A(2),\mathbf{e}_{3})$ and, thus, $\mathbf{e}_{3}$ is efficient for $A(2).$
\end{proof}

Note that, if the Perron vector of $A$ is efficient for $A$, Theorem \ref{t5}
also follows from the results in \cite{FJ2} (see Section \ref{sdig}).

\begin{corollary}
\label{c27}Let $A\in\mathcal{PC}_{4}$ and $w$ be the Perron vector of $A.$
Then $w$ is inefficient for $A\ $if and only if $G(A,w)$ has a sink vertex.
\end{corollary}

\begin{proof}
The "if" claim follows from Lemma \ref{blanq}. The "only if" claim follows
from Theorem \ref{t5} and Corollary \ref{ceff}.
\end{proof}

\bigskip

Any reciprocal matrix is similar, via a positive diagonal matrix, to a
reciprocal matrix with constant row sums (Lemma \ref{ldiag}). Thus, from
Corollary \ref{c27} (and Lemma \ref{lsim}), it follows that the matrices in
$\mathcal{PC}_{4}$ whose Perron vector is inefficient are those that are
diagonal similar, via a positive diagonal matrix, to a reciprocal matrix with
constant row sums and with a row in which all the off-diagonal entries are
greater than $1.$ We formalize this next.

\bigskip

Denote by $\mathcal{S}_{4}$ the set of matrices $A\in\mathcal{PC}_{4}$ with
Perron vector $\mathbf{e}_{4}$ (that is, with constant row sums) and such that
there is a row in which all the off-diagonal entries are greater than $1.$
Denote by $\mathcal{D}_{4}$ the set of positive diagonal matrices of size $4.$

\begin{theorem}
\label{th4by4}The Perron vector of $A\in\mathcal{PC}_{4}$ is inefficient for
$A$ if and only if $A=D^{-1}BD$ for some $D\in\mathcal{D}_{4}$ and some
$B\in\mathcal{S}_{4}.$
\end{theorem}

We next show that the example presented in \cite{blanq2006} of a matrix in
$\mathcal{PC}_{4}$ with inefficient Perron vector follows from Theorem
\ref{th4by4}.

\begin{example}
Let%
\[
A=\left[
\begin{array}
[c]{cccc}%
1 & 2 & 6 & 2\\
\frac{1}{2} & 1 & 4 & 3\\
\frac{1}{6} & \frac{1}{4} & 1 & \frac{1}{2}\\
\frac{1}{2} & \frac{1}{3} & 2 & 1
\end{array}
\right]  .
\]
The Perron vector of $A$ is $\allowbreak w=\left[
\begin{array}
[c]{cccc}%
2.9038 & 2.057 & 0.48282 & 1
\end{array}
\right]  ^{T}.$ For $D^{-1}=\operatorname*{diag}(w),$ we have
\[
B=DAD^{-1}=\left[
\begin{array}
[c]{cccc}%
1 & 1.4168 & 0.997\,64 & 0.68876\\
0.70584 & 1 & 0.938\,89 & 1.4585\\
1.0024 & 1.065\,1 & 1 & 1.0356\\
1.4519 & 0.68565 & 0.96563 & 1
\end{array}
\right]  .
\]
The matrix $B$ has Perron vector $\mathbf{e}_{4}$ and all the off-diagonal
entries in the third row of $B$ are greater than $1$ (vertex $3$ is a sink
vertex of $G(A,w)=G(B,\mathbf{e}_{4})$). Thus, the Perron vector of $A$ is
inefficient for $A.$
\end{example}

\section{Extending a reciprocal matrix to one with an efficient Perron
vector\label{s6}}

We next give a constructive proof of the existence of a matrix in
$\mathcal{PC}_{n}$ with efficient Perron vector, extending any given matrix in
$\mathcal{PC}_{n-1}.$ We need the following result that is used in the
construction of such an extension.

\begin{lemma}
\label{lwb}Let $B\in\mathcal{PC}_{k}$. Then, there is a positive diagonal
matrix $D\in M_{k}$ such that $DBD^{-1}$ is well-behaved and $\mathbf{e}_{k}$
is efficient for $DBD^{-1}$.
\end{lemma}

\begin{proof}
Let $D^{-1}=\operatorname*{diag}(b_{j}),$ in which $b_{j}$ is the $j$th column
of $B.$ Then, the $j$th column and the $j$th row of $B^{\prime}=DBD^{-1}$ are
$\mathbf{e}_{k}$ and $\mathbf{e}_{k}^{T}$, respectively$.$ This implies that
$\mathbf{e}_{k}$ is efficient for $B^{\prime},$ as any column of $B^{\prime}$
is efficient for $B^{\prime}$ (see \cite{FJ1}).

Let $r_{1}\geq$ $\cdots\geq r_{k}$ be the row sums of $B^{\prime}$. Since the
sum of the entries in the $j$th row of $B^{\prime}$ is $k,$ we have $r_{k}\leq
k.$ Suppose that $r_{1}-r_{k}<1.$ Then, for $i=1,\ldots,k,$ we have $0\leq
r_{i}-r_{k}<1,$ implying%
\[
\frac{1}{1+r_{k}-r_{i}}\geq1.
\]
Thus,%
\[
\sum_{i=1}^{k}\frac{1}{1+r_{k}-r_{i}}\geq k\geq r_{k},
\]
implying $f(1+r_{k}-r_{1})\geq0$, in which $f$ is as in (\ref{ff1}). Thus,
$B^{\prime}$ is well-behaved.
\end{proof}

\begin{theorem}
\label{t7}Let $B\in\mathcal{PC}_{n-1}$. Then, there is $A\in\mathcal{PC}_{n}$
with efficient Perron vector and such that $A(n)=B$.
\end{theorem}

\begin{proof}
By Lemma \ref{lwb}, there is a positive diagonal matrix $D$ such that
$B^{\prime}=DBD^{-1}$ is well-behaved and $\mathbf{e}_{n-1}$ is efficient for
$B$. With a possible permutation similarity on $B^{\prime}$, we assume that
$B^{\prime}$ has nonincreasing row sums $r_{1}\geq\cdots\geq r_{n-1}.$ Taking
into account Lemmas \ref{ldiagon} and \ref{lsim}, we consider $B=B^{\prime}.$
Let $A\in\mathcal{PC}_{n}$ be such that $A(n)=B$ and its last column is as in
(\ref{col}) with $x>0$ satisfying (\ref{eq1}). Note that the existence of (a
unique) such $x$ is ensured by Lemma \ref{lauxder}. Then, $\mathbf{e}_{n}$ is
the Perron vector of $A.$ Since $B$ is well-behaved and $\mathbf{e}_{n-1}$ is
efficient for $B,$ by Theorem \ref{t2}, $\mathbf{e}_{n}$ is efficient for $A.$
\end{proof}

\bigskip

Next we construct a matrix in $\mathcal{PC}_{4}$ with efficient Perron vector
and a prescribed principal submatrix in $\mathcal{PC}_{3}$. We consider the
case in which the well-behaved matrix appearing in the construction is of type I.

\begin{example}
Let $B$ be the matrix in (\ref{A4}), which is not well-behaved. We construct
$A\in\mathcal{PC}_{4}$ with efficient Perron vector and such that $A(4)=B$.
Recall that in Example \ref{exx0} the matrix $B$ was extended to a matrix in
$\mathcal{PC}_{4}$ with inefficient Perron vector. Let $D^{-1}%
=\operatorname*{diag}\left(  \frac{1}{5},1,\frac{9}{2}\right)  .$ The matrix%
\[
B^{\prime}=DBD^{-1}=\left[
\begin{array}
[c]{ccc}%
1 & 1 & 114.75\\
1 & 1 & 1\\
0.008715 & 1 & 1
\end{array}
\right]
\]
is well-behaved of type I and $\mathbf{e}_{3}$ is efficient for $B^{\prime}.$
Let\allowbreak\
\[
A^{\prime}=\left[
\begin{array}
[c]{cccc}%
1 & 1 & 114.75 & x\\
1 & 1 & 1 & x+114.75-1\\
0.008715 & 1 & 1 & x+114.75-0.008715\\
\frac{1}{x} & \frac{1}{x+114.75-1} & \frac{1}{x+114.75-0.008715} & 1
\end{array}
\right]  ,
\]
with
\[
x+2+114.75=1+\frac{1}{x}+\frac{1}{x+114.75-1}+\frac{1}{x+114.75-0.008715}.
\]
A calculation gives $x=0.00864,$ implying that
\[
A^{\prime}=\left[
\begin{array}
[c]{cccc}%
1 & 1 & 114.75 & 0.00864\\
1 & 1 & 1 & 113.76\\
0.008715 & 1 & 1 & 114.75\\
115.74 & 0.00879 & 0.008715 & 1
\end{array}
\allowbreak\right]  \allowbreak.\allowbreak
\]
The vector $\mathbf{e}_{4}$ is the Perron vector of $A^{\prime}$ and is
efficient for $A^{\prime}.$ Thus,
\begin{align*}
A  &  =\left(  D^{-1}\oplus\left[  1\right]  \right)  A^{\prime}\left(
D\oplus\left[  1\right]  \right) \\
&  =\left[
\begin{array}
[c]{cccc}%
1 & \frac{1}{5} & \smallskip\frac{51}{10} & 0.001728\\
5 & 1 & \smallskip\frac{2}{9} & 113.76\\
\smallskip\frac{10}{51} & \frac{9}{2} & 1 & 516.38\\
578.7 & 0.00879 & 0.001937 & 1
\end{array}
\allowbreak\right]
\end{align*}
satisfies $A(4)=B$ and the Perron vector of $A$ is efficient for $A.$
\end{example}

\bigskip

Given $B\in\mathcal{PC}_{n-1},$ we have shown how to construct a matrix
$A\in\mathcal{PC}_{n}$ with efficient Perron vector and such that $A(n)=B.$ If
$B\in\mathcal{PC}_{k},$ with $k<n-1,$ and we want to construct $A\in
\mathcal{PC}_{n}$ with efficient Perron vector and such that $A[\{1,\ldots
,k\}]=B,$ we may consider an arbitrary matrix $S\in\mathcal{PC}_{n-1}$ with
$S[\{1,\ldots,k\}]=B$, and proceed as above to obtain $A$ with efficient
Perron vector and such that $A(n)=S$.

\section{Conclusions\label{s8}$\allowbreak$}

When prioritizing alternatives, one important property that the weight vector
extracted from the reciprocal matrix, of the pair-wise comparisons, should
have is efficiency. One of the most used weighting methods employs the right
Perron vector of the reciprocal matrix as the vector of weights. It is known
that such a vector may not be efficient.

The reciprocal matrix constructed in practice from which a weight vector is
obtained may be just partially known. Here we study the existence of an
extension of a reciprocal matrix with efficient/inefficient Perron vector. We
conclude that any reciprocal matrix can be extended to one with inefficient
Perron vector, except if it is consistent of size $n-1,$ in which case it is
shown that there is no reciprocal extension of size $n$ with inefficient
Perron vector. We also show that any reciprocal matrix can be extended to one
with efficient Perron vector. Our analysis gives a procedure to construct such
extensions. As a consequence, we obtain structured classes of reciprocal
matrices with inefficient Perron vector, of which the family presented in
\cite{bozoki2014} is a particular case.

We give sufficient conditions for the digraph $G(A,w)$ (see Section
\ref{sdig}), associated with an $n$-by-$n$ reciprocal matrix $A$ with
inefficient Perron vector $w$, to have a sink vertex, namely, the inefficient
Perron vector has a subvector obtained by deleting one entry that is either
the Perron vector of the corresponding principal submatrix of $A$ or is
efficient for this submatrix$.$ Though this latter condition does not always
happen for $n>4,$ as illustrated, it holds when $n=4.$ This implies that the
Perron vector $w$ of a $4$-by-$4$ reciprocal matrix $A$ is inefficient for
$A\ $if and only if the associated digraph $G(A,w)$ has a sink vertex. Several
examples illustrating the theoretical results are provided throughout the paper.

This work leaves some relevant questions for future study, such as determining
all extensions of a given reciprocal matrix with inefficient Perron vector and
if there are such extensions for which the associated digraph has a source
vertex (we conjecture the answer is negative). Also, studying the existence of
reciprocal completions with efficient (inefficient) Perron vector for other
patterns of the specified entries is an interesting problem to consider.

\bigskip

\textbf{Acknowledgement} We would like to thank Xingyu Wang, a student at the
College of William and Mary, who provided us Example \ref{ex4} and helpful
simulations about the efficiency of the Perron vector.

\end{document}